\documentclass{amsart}
 \usepackage{amssymb,stmaryrd}
 \usepackage{tikz, tikz-cd, pgfplots, amsfonts}
 \usepackage[hidelinks]{hyperref}

\newcommand{\textoverline}[1]{$\overline{\mbox{#1}}$}

 \newtheorem{theorem}{{\rm T\sc heorem}}[section]
 \newtheorem{lemma}[theorem]{{\rm L\sc emma}}
 
 \newtheorem{proposition}[theorem]{{\rm P\sc roposition}}

 \newtheorem{definition}[theorem]{{\rm D\sc efinition}}

 \theoremstyle{definition}
 \newtheorem{remark}[theorem]{{\rm R\sc emark}}
  
  \newtheorem*{acks}{Acknowledgments}

\DeclareMathOperator{\aut}{Aut}
\DeclareMathOperator{\pic}{Pic}
\DeclareMathOperator{\bl}{Bl}

\DeclareMathOperator{\mw}{MW}
\DeclareMathOperator{\ns}{NS}
\DeclareMathOperator{\numer}{Num}

\begin{document}
 \title {Severi varieties on Enriques surfaces of base change type and on rational elliptic surfaces}
 \author[Simone Pesatori]{Simone Pesatori}
 \address{Università degli Studi Roma Tre\\
        Largo San Leonardo Murialdo, 1\\
         Rome\\
         Italy}

 \begin{abstract} If an irreducible curve on the very general Enriques surface splits in the K3 cover, its preimage consists of two linearly equivalent irreducible curves. We prove the nonemptiness of countable families of Severi varieties of curves of any genus on Enriques surfaces of base change type, whose members split in nonlinearly equivalent curves in the K3 cover. Our machinery leads us to provide examples of special superabundant logarithmic Severi varieties of curves of any genus on rational elliptic surfaces.
\end{abstract}
 \maketitle
 \pagestyle{myheadings}\markboth{\textsc{ Simone Pesatori }}{\textsc{Severi varieties on Enriques surfaces of base change type and on rational elliptic surfaces}}

\section{Introduction}

This paper addresses the topic of Severi varieties of curves on Enriques surfaces of base change type and on rational elliptic surfaces.\par
Severi varieties parametrize curves with a prescribed number of nodes (or with a fixed geometric genus, or with given tangency conditions to a fixed curve) in a fixed linear system. Severi varieties were introduced by Severi in \cite{Se}, where he proved that all Severi varieties of irreducible $\delta$-nodal curves of degree $d$ in $\mathbb{P}^2$ are nonempty and smooth of the expected dimension. Their irreducibility was proved by Harris in \cite{H}. Severi varieties on other surfaces have received much attention in recent years, especially in connection with enumerative formulas computing their degrees (see for example \cite{B}, \cite{BOP} and \cite{CH}). Nonemptiness, smoothness, dimension and irreducibility for Severi varieties have been widely investigated on various rational surfaces (see, e.g., \cite{GLS}, \cite{Ta} and \cite{Te}), as well as K3, Enriques and abelian surfaces (see, e.g.,\cite{Che}, \cite{KL}, \cite{KLM}, \cite{LS}, \cite{BL}, \cite{MM}, \cite{CDGK} and \cite{CDGK2}). \par
It is well-known that every Enriques surface is equipped with several genus 1 pencils. Every genus 1 fibration $Y\rightarrow\mathbb{P}^1$ on an Enriques surface $Y$ gives rise to a genus 1 fibration $X\rightarrow\mathbb{P}^1$ on its K3 cover, which, in general, does not admit a section. \textit{Special} elliptic fibrations are characterized among elliptic fibrations of $Y$ by the fact that they admit a smooth rational bisection, which necessarily splits into two (disjoint) sections on $X$. More generally, Hulek and Sch\"utt in \cite{HS} studied the so-called $m$-special genus 1 fibrations on $Y$ that are characterized by the fact that they admit a (necessarily rational) bisection that splits into two sections on $X$ with intersection number $2m$. This construction yields countably many distinct divisors in the 10-dimensional moduli space of Enriques surfaces, whose members are called Enriques surfaces \textit{of base change type} or \textit{$m$-special} Enriques surfaces.\par In \cite{Pesa}, the author proves that the bisections that appear in the construction of Hulek and Sch\"utt are nodal for very general members of all the 9-dimensional families. Moreover, the author proves the existence of nodal bisections of arbitrarily large arithmetic genus on any fixed such surface.\par
 The construction of Hulek and Sch\"utt strongly involves some rational elliptic surfaces, which turn out to be covered by the same K3 surface which covers the $m$-special Enriques surface. In particular, we have the following diagram
 \begin{center}
                    \[\begin{tikzcd}[ampersand replacement=\&,row sep=large,column sep=huge]
                       \& X_m\arrow[dl] \arrow[dr]\\ S\&\&Y_m
                    \end{tikzcd}\]
\end{center} 
Here $S$ is a rational elliptic surface, $Y_m$ is an $m$-specal Enriques surface and $X_m$ is a K3 surface dominating both as double covering.\par
 The first result of this paper is the proof of the existence of some  particular equigeneric Severi varieties of curves on every $Y_m$, which we call \textit{special nonregular}. Their members split in the K3 cover in two nonlinearly equivalent curves. This cannot happen on the very general Enriques surface: as pointed out by Ciliberto, Dedieu, Galati and Knutsen in \cite{CDGK}, if an irreducible nodal curve on the very general Enriques surface splits in its universal cover, then its preimage necessarily consists of two linearly equivalent curves. After noticing that their result also holds for nonnodal curves, we prove that every $Y_m$ is provided with lots of such special varieties. More precisely, we prove the following theorem.
 \begin{theorem}\label{teorenr}
     For every $n\in\mathbb{N}$, every $m$-special Enriques surface $Y_m$ admits a countable number of families of special nonregular components of some Severi varieties of curves of genus $n$.
 \end{theorem} 
 
 The second result concerns the so-called \textit{logarithmic} Severi varieties of curves, which parametrize curves in a fixed linear system with given tangency conditions to a fixed curve. They have been studied by Dedieu in \cite{De} and his results are based on the works of Caporaso and Harris (see for example \cite{CH}). We study the case of rational elliptic surfaces and provide examples of what we call \textit{special superabundant} logarithmic Severi varieties of curves. They have dimension greater than expected and parametrize curves which are totally tangent to the branch locus of $X_m\rightarrow S$, splitting on $X_m$ in nonlinearly equivalent curves. The idea is similar to the one presented in \cite{De} and \cite{CD}, where Ciliberto and Dedieu provide examples of superabundant logarithmic Severi varieties on the plane exploiting double covers, but the curves on the Severi varieties they consider split in linearly equivalent curves on the surface upstair. Our result is given in the following theorem. 
 \begin{theorem}\label{teorres}
     For every $n\in\mathbb{N}$, the general rational elliptic surface admits a countable number of families of special superabundant components of some logarithmic Severi varieties of curves of genus $n$.
 \end{theorem}
 In un upcoming paper, we exploit our machinery to study in details the particular case of genus 1 fibrations on $X_m$ and their behaviour with respect to the Enriques involution and to the involution giving rise to the rational elliptic surface.\par
 
 In Section \ref{sec2} we recall the basics about the surfaces involved in the paper, such as K3 surfaces, Enriques surfaces (of base change type), rational elliptic surfaces and the genus 1 pencils lying on them. Furthermore, we analyze the behaviour of the curves on the K3 surface $X_m$ with respect to the Enriques involution and to the rational involution. In Section \ref{sec3} we prove the main results.

 \begin{acks}This work is the continuation of a part of my PhD Thesis and I want to thank my advisor Andreas Leopold Knutsen for the fruitful conversations which inspired this paper. \end{acks}

\section{Rational elliptic surfaces and Enriques surfaces of base change type\label{sec2}}

We recall the basics about elliptic surfaces, K3 surfaces and Enriques surfaces. As general references the reader might consult \cite{M}, \cite{BHPV} or \cite{CD}. \par

For us a \textit{genus one fibration} is a morphism $f:X\rightarrow C$, where $X$ is an algebraic surface and $C$ is a smooth curve, such that the general fiber is a smooth curve of genus one. If there is a section $s:C\rightarrow X$, we say that $f:X\rightarrow C$ is an \textit{elliptic fibration} (with a given section), and $X$ is an \textit{elliptic surface} over $C$.\par
If $X\rightarrow C$ is equipped with a chosen section $s$ the set of sections is an abelian group with the group addition defined fiberwise. The group of the sections of $X\rightarrow C$ is called Mordell-Weil group of the elliptic surface, denoted by $\mw(X\rightarrow C)$ or simply $\mw(X)$ if the surface has only one elliptic fibration or if it is clear what fibration we are referring to. The chosen section $s$, which is the zero element of $\mw(X)$, is called the zero-section. \par
By \cite[Lemma IV.1.2]{M}, every rational elliptic surface $S$ arises as the blow up of the projective plane $\mathbb{P}^2$ at nine points $P_1,\dots,P_9$ which are the base points of a pencil of cubic curves.\par
It is well-known that every rational elliptic surface has twelve singular fibers counted with multiplicity: in this work, unless differently specified, we will consider just rational elliptic surfaces with twelve nodal curves as singular fibers.
\begin{definition}
    A rational elliptic surface is \textit{general} if it has twelve nodal curves as singular fibers.
\end{definition}
We call $E_1,\dots,E_9$ the exceptional divisors over the points $P_1,\dots,P_9$. With this notation, the Picard group of $S\cong\bl_{\{P_1,\dots,P_9\}}\mathbb{P}^2$ is 
\begin{center}$\pic(S)\cong\mathbb{Z}L\oplus\mathbb{Z}E_1\oplus\dots\oplus\mathbb{Z}E_9$,\end{center}
with $L$ the pullback of a line in $\mathbb{P}^2$. We choose the last exceptional divisor $E_9$ to be the zero-section of the fibration. 
 With the previous notations, the Mordell-Weil group of  $S$ is \begin{center} $\mw(S)\cong\mathbb{Z}^8$, \end{center}and it is generated by the exceptional divisors $E_1,\dots,E_8$. The neutral element is the zero-section $E_9$.\par
We recall the basics about K3 surfaces and Enriques surfaces. As general references the reader might consult \cite{BHPV} or \cite{CD}. 

\begin{definition}
    A smooth projective surface $X$ is called \textit{K3 surface}
if $X$ is simply connected with trivial canonical bundle $\omega_X\cong\mathcal{O}_X$.\\
An \textit{Enriques surface} $Y$ is a quotient of a K3 surface $X$ by a fixed point free involution. Such an involution is also called \textit{Enriques involution}.
\end{definition}
For any Enriques surface $Y$ there is two-torsion in $\ns(Y)$ represented by the canonical divisor $K_Y$. The quotient $\ns(Y)_f$ of $\ns(Y)$ by its torsion subgroup is an even unimodular lattice, which is isomorphic to the so-called \textit{Enriques lattice}: \begin{center}
    $\numer(Y)=\ns(Y)_f\cong U\oplus E_8(-1)$.
\end{center}
The Picard rank of any Enriques surface is equal to 10.

\begin{definition}\label{verygenenr}
    We say that an Enriques surface $Y$ is \textit{Picard very general} if its universal cover $X$ is such that \begin{center}
        $\ns(X)\cong U(2)+E_8(-2)$.
    \end{center}
\end{definition}
\begin{remark}
    An Enriques surface $Y$ is Picard very general if and only if the Picard rank of its universal cover $X$ is equal to $10$.
\end{remark}
K3 and Enriques surfaces are the only surfaces that may admit more than one genus 1 pencil.
It is well-known that every genus 1 pencil on an Enriques surface has exactly two fibers of multiplicity two, called \textit{half-fibers}. The canonical divisor of $Y$ can be represented as the difference of the supports of the half-fibers of a genus 1 pencil: if $F$ is a genus 1 pencil of $Y$ and \begin{center}
    $2E_1\sim F$ and $2E_2\sim F$, then $K_Y\sim E_1-E_2$.
\end{center}
Ohashi in \cite{Oa} has classified K3 surfaces with Picard number $\rho=11$ and an Enriques involution. The Enriques quotient of such K3 surfaces are not Picard very general. Hulek and Sch\"utt in \cite{HS} show a construction of one type of such K3 surfaces, obtained starting from a rational elliptic surface. We are going to review their construction. 
\subsection{\rm E\sc nriques surfaces of base change type}
Let $S=\bl_{\{P_1,\dots,P_9\}}\mathbb{P}^2$ denote a general rational elliptic surface. We let now \begin{center}
    $g:\mathbb{P}^1\rightarrow\mathbb{P}^1$
\end{center} be a morphism of degree two. Denote the ramification points by $t_0$ and $t_{\infty}$. It is well-known that the pullback of $X$ via $g$ is a K3 surface under the assumption that the fibers of $S$ over $t_0$ and $t_{\infty}$ are smooth. With abuse of notation, we denote by $g$ also the double cover \begin{center} $X\rightarrow S$  \end{center} and we denote by $\tilde{E}_i$ the pull-backs of the exceptional divisors $E_i$ of $S$: with this notation, $\tilde{E}_9$ is the zero-section for the induced elliptic fibration on $X$. \par Moreover, we denote by $S_t$ the fiber on $S$ over a point $t\in\mathbb{P}^1$ and by $X_t$ and $X_{-t}$ the two components of its preimage on $X$ (if $t\neq t_0,t_{\infty}$). Since $S_t\cong X_t\cong X_{-t}$, if $Q_t\in S_t$, we denote the two points in its preimage $g^{-1}(Q_t)$ by $\tilde{Q_t}$ and $\tilde{Q}_{-t}$ respectively. Sometimes, we refer to the pair $X_t$ and $X_{-t}$ as \textit{twin fibers}, to the pair $\tilde{Q_t}$ and $\tilde{Q}_{-t}$ as \textit{twin points in twin fibers} and to the pair $\tilde{Q_t}$ and $\boxminus\tilde{Q}_{-t}$ as \textit{opposite points in twin fibers} (with respect to $\tilde{E}_9$). \par
Let $\iota$ denote the deck transformation for $g$, i.e. $\iota\in\aut(\mathbb{P}^1)$ such that $g=g\circ\iota$. 
Then $\iota$ induces an automorphism of $X$ that we shall also denote by $\iota$. The quotient $X/\iota$ returns exactly the rational elliptic surface $S$ we started with. 

\begin{remark}
    We obtain a ten-dimensional family of elliptic K3 surfaces: eight dimensions from the rational elliptic surfaces and two dimensions from the choice of the ramification for the base change. 
\end{remark}

\begin{definition}\label{vergenbaschan}
    We say that such a K3 surface $X$ \textit{base change very general} if $S$ is general, $S_{t_0}$ and $S_{t_{\infty}}$ are smooth elliptic curves and $\iota^*$ acts as the identity on $\ns(X)$.
\end{definition}

\begin{remark}
    As pointed out in \cite[Section 3.2]{HS}, a base change very general K3 surface $X$ does not carry any Enriques involution.
\end{remark}

 In \cite{HS}, Hulek and Sch\"utt impose a geometric condition on the base change $g:X\rightarrow S$ that allows to construct (a countable number of) families of K3 surfaces with an Enriques involution.
 In order to exhibit K3 surfaces with Enriques involution within our family of K3 surfaces of base change type, we need the following lemma, that summarizes the discussion in \cite[Section 3.3]{HS}. This generalizes the original result given by Kond\textoverline{o} in \cite[Lemma 2.6]{Ko} for nodal Enriques surfaces.\par
We say that a section $P$ on the elliptic fibration on $X$ is anti-invariant with respect to $\iota^*$ if $\iota^*(P)=(-1)^*(P)$, where $(-1)$ indicates the involution on $X$ acting fiberwise by interchanging opposite points with respect to the zero-section $\tilde{E}_9$.
\begin{lemma}[Hulek-Sch\"utt]
    Let $P$ be a section for the elliptic fibration on $X$ induced by the one of $S$. Then \begin{itemize}
        \item either $P$ is invariant with respect to $\iota^*$,
        \item or $P$ is anti-invariant with respect to $\iota^*$.
    \end{itemize}
    Moreover, in the former case, $P$ is the pull-back of a section $E\in\mw(S)$ and it cuts twin points in twin fibers, while in the latter $P$ cuts opposite points in twin fibers.
\end{lemma}

 We denote by $\boxplus P\in\aut(X)$ the automorphism on $X$ given by the fiberwise translation for $P$.

 Consider the following automorphism of $X$ \begin{center}
    $\tau:=\iota\circ (\boxminus P)$.
\end{center}
\begin{proposition}[Hulek-Sch\"utt]\label{enrinv}
    $\tau\in\aut(X)$ is an involution and it is an Enriques involution if and only if $P$ does not intersect $\tilde{E}_9$ along $X_{t_0}$ and $X_{t_{\infty}}$.
\end{proposition}

\begin{definition}We denote by $Y=X/\tau$ the Enriques surface obtained with the construction described in Theorem \ref{enrinv}. We say that $Y$ is an \textit{Enriques surfaces of base change type} and we denote by $f$ the quotient $X\rightarrow Y$.  \end{definition}

\begin{remark}The given elliptic fibration on $X$ induces a genus 1 fibration on $Y$. Here the smooth fiber $Y_t$ of $Y$ at $t$ is isomorphic to the fibers $X_t$ and $X_{-t}$ at $g^{-1}(t)$ as genus 1 curves or to the fiber of the rational elliptic surface $S_t$.\end{remark}

The proof of existence of anti-invariant sections is performed by Hulek and Sch\"utt in a lattice-theoretical way (see \cite[Section 3]{HS}).
The following theorem collects their results that we shall need in this work.
 \begin{theorem}[Hulek, Sch\"utt]\label{hs}
    For every nonnegative integer $m\in\mathbb{Z}_+$, there exists an irreducible nine-dimensional family $\mathcal{F}_m$ of K3 surfaces such that, for every $X_m\in\mathcal{F}_m$, \begin{center}
    $\mw(X_m)\cong\mathbb{Z}^9$ and $\rho(X_m)=11$. 
    \end{center}Moreover, $X_m$ covers an Enriques surface $Y_m$ and a rational elliptic surface $S$. Finally, $X_m$ admits an anti-invariant section $R_m$ with respect to $\iota^*$ such that $R_m\cdot\tilde{E}_9=2m$, where $\tilde{E}_9$ is the pullback of the zero-section $E_9\in\mw(S)$.
\end{theorem}
\begin{proof}
    See \cite[Section 3]{HS}.
\end{proof}
We shall say that $Y_m$ is an Enriques surface of base change type (or an $m$-special Enriques surfaces). The family $\mathcal{F}_0$ turns out to be precisely the nine-dimensional family of K3 surfaces covering the Enriques surfaces which admit smooth $(-2)$-curves (called \textit{special} or \textit{nodal}).\par
Notice that the Picard number of a Picard very general K3 surface is 10 and the rank of the Mordell-Weyl group of a base change very general K3 surface is 8. For $X_m\in\mathcal{F}_m$, both the numbers increase by 1. This is due to the presence of the "extra-section" $R_m$ for the elliptic pencil on $X_m$ induced by the one of $S$. We have the following diagram
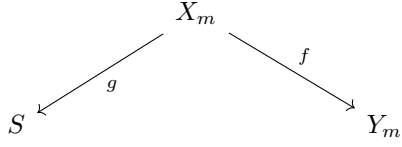
\begin{center}
                    \[\begin{tikzcd}[ampersand replacement=\&,row sep=large,column sep=huge]
                       \& X_m\arrow[dl,"g"] \arrow[dr,"f"]\\ S\&\&Y_m
                    \end{tikzcd}\]
\end{center} 

We denote by $B_{S,m}$ the image of $R_m$ on $S$ under $g$. In \cite{Pesa}, the author prove the following theorem.
\begin{theorem}\label{bybs}
   The curve $B_{S,m}$ is a nodal rational bisection for the elliptic pencil on $S$ and its linear class on $\pic(S)$ is \begin{center} $B_{S,m}\sim 6(m+1)L-2(m+1)E_1-\dots-2(m+1)E_8-2mE_9$. \end{center}
    Moreover, $B_{S,m}$ is tangent to the fixed fibers $S_{t_0}$ and $S_{t_{\infty}}$ and $g^*(B_{S,m})=R_m+\boxminus R_m$.
\end{theorem}

The following lemma ensure that the extra-section $R_m$ is not a torsion section for the elliptic fibration on $X_m$.

\begin{lemma}\label{inford}
    $R_m$ has infinite order in $\mw(X_m)$.
\end{lemma}
\begin{proof}
We have that 
    \begin{center} $B_{S,m}^2=\frac{1}{2}(R_m+\boxminus R_m)^2=-2+R_m\cdot\boxminus R_m$ \end{center} and \begin{center}
        $B_{S,m}^2=(6(m+1)L-2(m+1)E_1-\dots-2(m+1)E_8-2mE_9)^2=8m+4$,
    \end{center} from which $R_m\cdot\boxminus R_m=8m+6$.
    It is a very well-known fact that any two torsion sections of a smooth relatively minimal elliptic surface are disjoint (see, for example, \cite[Proposition VII.3.2]{M}). Let us assume $R_m$ of $l$-torsion for some $l\in\mathbb{N}$: since the opposite of an $l$-torsion point is an $l$-torsion point, also $\boxminus R_m$ is an $l$-torsion section. But $R_m\cdot\boxminus R_m=8m+6$, whence we have a contradiction. 
\end{proof}

\subsection{\rm B\sc ehaviour of the curves on $X_m$}
Recall that we have the following diagram 
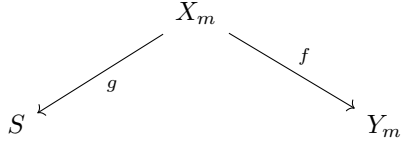
\begin{center}
                    \[\begin{tikzcd}[ampersand replacement=\&,row sep=large,column sep=huge]
                       \& X_m\arrow[dl,"g"] \arrow[dr,"f"]\\ S\&\&Y_m
                    \end{tikzcd}\]
\end{center} 
The next two propositions describe the behaviour of the curves on $X_m$ with respect to the Enriques quotient and to the rational quotient. They will be our main tool to prove Theorem \ref{teorenr} and Theorem \ref{teorres} \par

We denote by $R_m^{\boxplus k}:=\underbrace{R_m\boxplus\dots\boxplus R_m}_{k\textit{ times}}$ the sum of the section $R_m$ with itself $k$ times in $\mw(X_m)$. Moreover, given a curve $C\subset S$, we denote by $C_{X,m}^{\boxplus k}$ the curve $g^{-1}(C)\boxplus R_m^{\boxplus k}$, by $C_S^{\boxplus k}$ its image $g(C_{X,m}^{\boxplus k})\subset S$ and by $C_Y^\boxplus k$ its image $f(C_{X,m}^{\boxplus k})\subset Y$. 

\begin{proposition}\label{partenzares}
    Let $C\subset S$ be any effective curve on $S$ different from an elliptic fiber. Then, the curve $C_{X,m}^{\boxplus k}\subset X_m$ is identified by $g$ with $C_{X,m}^{\boxminus k}$ and by $f$ with $C_{X,m}^{\boxplus (1-k)}$.
\end{proposition}

\begin{proof}
    $C_X$ cuts twin points in opposite fibers: indeed, if \begin{center} $C_{|S_t}=a_1Q_{1,t}+\dots+a_rQ_{r,t}$,\end{center} then \begin{center} ${C_X}_{|X_t}=a_1Q_{1,t}+\dots+a_rQ_{r,t}$\end{center} and \begin{center} ${C_X}_{|X_{-t}}=a_1Q_{1,-t}+\dots+a_rQ_{r,-t}$.\end{center} 
   Now,\begin{center} ${C_{X,m}^{\boxplus k}}_{|X_t}=a_1(Q_{1,t}\boxplus R_{m_{t}}^{\boxplus k})+\dots+a_r(Q_{r,t}\boxplus R_{m_{t}}^{\boxplus k})$: \end{center} for every $j\in\{1,\dots,r\}$, we have \begin{center} $\iota(Q_{j,t}\boxplus R_{m_{t}}^{\boxplus k})=(Q_{j,-t}\boxminus R_{m_{-t}}^{\boxplus k})\in C_{X,m}^{\boxminus k}$. \end{center} In the same way one proves that $\iota$ sends every point of $C_{X,m}^{\boxminus k}$ to points belonging to $C_{X,m}^{\boxplus k}$.\par
Regarding the Enriques involution, for every $j\in\{1,\dots,r\}$ we have \begin{center} $\tau(Q_{j,t}\boxplus R_{m_{t}}^{\boxplus k})=(\iota\circ\boxminus R_m)(Q_{j,t}\boxplus R_{m_{t}}^{\boxplus k})=\iota(Q_{j,t}\boxminus R_{m_{t}}^{\boxplus (k-1)})=Q_{j,-t}\boxplus R_{m_{-t}}^{\boxplus (1-k)}\in C_{X,m}^{\boxplus (1-k)}$ \end{center}
and \begin{center} $\tau(Q_{j,t}\boxplus R_{m_{t}}^{\boxplus (1-k)})=(\iota\circ\boxminus R_m)(Q_{j,t}\boxplus R_{m_{t}}^{\boxplus (1-k)})=\iota(Q_{j,t}\boxplus R_{m_{t}}^{\boxminus k})=(Q_{j,-t}\boxplus R_{m_{-t}}^{\boxplus k})\in C_{X,m}^{\boxplus k}$. \end{center}

\end{proof}
In particular, the preimage $C_X:=g^{-1}(C)$ doubly covers $C$ and it is identified by $f$ with $C_{X,m}:=C_X\boxplus R_m$.\par

The next diagram summarizes the result given in Proposition \ref{partenzares}

\begin{center}
                    \[\begin{tikzcd}[ampersand replacement=\&,row sep=large,column sep=huge]
                       \& \textcolor{red}{C_{X,m}^{\boxplus k}\cup C_{X,m}^{\boxminus k}}\subset X_m\supset \textcolor{blue}{C_{X,m}^{\boxplus k}\cup C_{X,m}^{\boxplus (1-k)}} \arrow[dl,"g"] \arrow[dr,"f"]\\ \textcolor{red}{C_S^{\boxplus k}}\subset S\&\&Y_m\supset\textcolor{blue}{C_Y^{\boxplus k}}
                    \end{tikzcd}\]
\end{center} 
If we consider the pullback of a curve on the Enriques surface $Y_m$, we obtain a quite symmetric result. For a curve $L\subset Y_m$, we denote by $L_{X,m}^{\boxplus k}$ the curve $f^{-1}(L)\boxplus R_m^{\boxplus k}$, by $L_Y^{\boxplus k}$ its image $f(L_{X,m}^{\boxplus k})\subset S$ and by $L_S^{\boxplus k}$ its image $g(C_{X,m}^{\boxplus k})\subset Y$. 

\begin{proposition}\label{partenzaenr}
    Let $L\subset Y_m$ be any effective curve on $Y_m$ different from a fiber or an half-fiber of $\epsilon_Y$. Then, the curve $L_{X,m}^{\boxplus k}$ is identified by $f$ with $L_{X,m}^{\boxminus k}$ and by $g$ with $L_{X,m}^{\boxminus (k+1)}$. 
\end{proposition}
\begin{proof}
    As in the proof of Proposition \ref{partenzares}, for every $t\in\mathbb{P}^1$ we set \begin{center} ${L_X}_{| X_t}=a_1Q_{1,t}+\dots+a_rQ_{r,t}$. \end{center} Since $\tau^*L_X=L_X$, with $L_X:=f^{-1}(L)$, and \begin{center}
        $\tau(a_jQ_{j,t})=a_jQ_{j,t}\boxplus R_{m.t}$,
    \end{center} we have \begin{center}
    ${L_X}_{|X_{-t}}=a_1Q_{1_{-t}}\boxplus R_{m_{-t}}+\dots+a_rQ_{r_{-t}}\boxplus R_{m_{-t}}$.
    \end{center} Now, \begin{center} ${L_X^{\boxplus k}}_{|X_{t}}=a_1Q_{1_{t}}\boxplus R_{m_{t}}^{\boxplus k}+\dots+a_rQ_{r_{t}}\boxplus R_{m_{t}}^{\boxplus k}$ \end{center}and \begin{center}
        $\tau(a_jQ_{j_t}\boxplus R_{m_t}^{\boxplus k})=a_jQ_{j_{-t}}\boxplus R_{m_{-t}}^{\boxplus (1-k)}$.
    \end{center}Since \begin{center}
    ${{L_X^{\boxminus k}}_{|X_{-t}}}=(a_1Q_{1_{-t}}\boxplus R_{m_{-t}})\boxminus R_{m_t}^{\boxplus k}+\dots+(a_rQ_{r_{-t}}\boxplus R_{m_{-t}})\boxminus R_{m_t}^{\boxplus k}=a_1Q_{1_{-t}}\boxplus R_{m_{-t}}^{\boxplus (1-k)}+\dots+a_rQ_{r_{-t}}\boxplus R_{m_{-t}}^{\boxplus (1-k)}$,\end{center}  we conclude that \begin{center}
        $\tau^*(L_X^{\boxplus k})=L_X^{\boxminus k}$.
    \end{center}
     With an analogous argument, one proves that \begin{center}
         $\iota^*(L_X^{\boxplus k})=L_X^{\boxminus (k+1)}$
     \end{center}
    
\end{proof}
In particular, the preimage $L_X\subset X_m$ doubly covers $L$ and it is identified by $g$ with $L_X\boxminus R_m$. The following diagram describes the behaviour of pullbacks under $f$ of smooth curves on $Y_m$ with respect to $\iota$ and $\tau$.
\begin{center}
                    \[\begin{tikzcd}[ampersand replacement=\&,row sep=large,column sep=huge]
                       \& \textcolor{red}{L_{X,m}^{\boxplus k}\cup L_{X,m}^{\boxminus (k+1)}}\subset X_m\supset \textcolor{blue}{L_{X,m}^{\boxplus k}\cup L_{X,m}^{\boxminus k}} \arrow[dl,"g"] \arrow[dr,"f"]\\ \textcolor{red}{L_S^{\boxplus k}}\subset S\&\&Y_m\supset\textcolor{blue}{L_Y^{\boxplus k}}
                    \end{tikzcd}\]
\end{center} 

\section{Severi varieties}\label{sec3}

Let $S$ be a smooth complex projective surface and $L$ a line bundle on $S$ such that the complete linear system $|L|$ contains smooth, irreducible curves (such a line bundle,
or linear system, is often called a Bertini system). Let
\begin{center}
    $p:=p_a(L)=\frac{1}{2}L\cdot(L+K_S)+1$
\end{center} be the arithmetic genus of any curve in $|L|$.
We will consider three different kind of Severi varieties of curves on surfaces: one parametrizing irreducible curves with a prescribed number of nodes, another parametrizing irreducible curves with a fixed geometric genus and the last parametrizing irreducible curves with given tangency conditions to a fixed curve.
\begin{definition}\label{severi}
    For any integer $0\leq\delta\leq p$, consider the locally closed, functorially defined subscheme of $|L|$ \begin{center}
        $V_{|L|,\delta}(S)$ or simply $V_{|L|,\delta}$
    \end{center}parametrizing irreducible curves in $|L|$ having only $\delta$ nodes as singularities: this is called the \textit{Severi variety} of $\delta$-nodal curves in $|L|$. We will let $g:=p-\delta$ be the geometric genus of the curves in $V_{|L|,\delta}$.
\end{definition}
\begin{definition}
    For any given integer $g$ such that $0\leq g\leq p_a(L)$, the locally closed subscheme of $|L|$ \begin{center} $V^{|L|}_g(S)$ or simply $V^{|L|}_g$ \end{center} whose geometric points parametrize reduced and irreducible curves $C$
having geometric genus $g$, i.e. such that their normalizations have genus $g$. We shall call such a family the \textit{equigeneric Severi variety} of genus $g$ curves in $|L|$.
\end{definition}
When $\delta=p_a(L)-g$, we have $V_{|L|,\delta}\subset V^{L}_g$.\par
It is well-known that, if $V_{|L|,\delta}$ is nonempty, then all of its irreducible components $V$ have dimension $\dim(V)\geq\dim|L|-\delta$. If $V_{|L|,\delta}$ is smooth of dimension $\dim|L|-\delta$ at $[C]$ it is said to be \textit{regular at $[C]$}. An irreducible component $V$ of $V_{|L|,\delta}$ will be said to be \textit{regular} if the condition of regularity is satisfied at any of its points, equivalently, if it is smooth of dimension $\dim|L|-\delta$.\par

\subsection{\rm S\sc pecial nonregular components of equigeneric \rm S\sc everi varieties}

We defined the regular components of Severi varieties of $\delta$-nodal curves on an algebraic surface to be the smooth components of the expected dimension. We shall call \textit{nonregular} the other ones. Regarding the Enriques surfaces, in \cite[Proposition 1]{CDGK} Ciliberto, Dedieu, Galati and Knutsen prove that the regular components have dimension $g-1$, while the nonregular ones have dimension $g$, where $g$ is the geometric genus of the general curve of the Severi variety. Moreover, the authors show that the regular components parametrize curves with irreducible preimage in the K3 cover, and the nonregular components parametrize curves which split in the K3 cover: if a component of a Severi variety of $\delta$-nodal curves on a very general Enriques surface is nonregular, then its members split in two linearly equivalent curves in the K3 cover.\par
We show that the assumption of very generality is necessary: in fact, in the Enriques surfaces of base change type (that form a countable set of 9-dimensional subfamilies in the moduli space of Enriques surfaces), there are plenty of nonregular Severi varieties violating the result of Ciliberto, Dedieu, Galati and Knutsen. 
\begin{definition}
    We call \textit{special nonregular component} a nonregular component of a Severi variety of curves on an Enriques surface such that its members split in two nonlinearly equivalent curves in the K3 cover. We call \textit{special nonregular curve} (nonregular) every member of a special nonregular (nonregular) component.
\end{definition}
The first example of special nonregular curves are the $m$-special curves on the Enriques surfaces of base change type.
\begin{proposition}\label{bynongen}
    Let $B_{Y,m}$ be an $m$-special curve on a general $m$-special Enriques surface $Y_m$. Then, $B_{Y,m}$ is a special nonregular curve. 
\end{proposition}
\begin{proof}
     First of all, every nodal rational curve is nonregular: the geometric genus of the rational curves is 0, then they necessarily belong to nonregular components. As proved in \cite[Theorem 1.1]{Pesa}, $B_{Y,m}$ is nodal, whence it is nonregular. Now, $B_{Y,m}$ splits in two $(-2)$-curves (which are not linearly equivalent) in the K3 cover and therefore it is a special nonregular curve.
\end{proof}
The result given by Ciliberto, Dedieu, Galati and Knutsen in \cite{CDGK} does not say anything about the equigeneric Severi varieties of curves on $Y$. Let $C$ be an irreducible curve on an Enriques surface $Y$.
We denote by $\nu_C:\overline{C}\rightarrow C$ the normalization of $C$ and define $\eta_C:=\mathcal{O}_C(K_S)=\mathcal{O}_C(-K_S)$, a nontrivial 2-torsion element in $\pic^0(C)$, and $\eta_{\overline{C}}:=\nu^*\eta_C$. By standard results on covering of complex manifolds (see \cite[Section 2]{CDGK} or \cite[Section 1.17]{BHPV}), two cases may happen: \begin{itemize}
    \item $\eta_{\overline{C}}\ncong\mathcal{O}_{\overline{C}}$ and $f^{-1}(C)$ is irreducible;
    \item $\eta_{\overline{C}}\cong\mathcal{O}_{\overline{C}}$ and $f^{-1}(C)$ consists of two irreducible components conjugated by the Enriques involution. 
\end{itemize} 
Following \cite[proofs of Thm 4.2 and Cor 2.7]{DS}, the dimension of the equigeneric Severi variety of genus
$g$ curves in $|L|=|C|$ at the point $|C|$ satisfies the inequality \begin{center}
    $\dim_{[C]}V^{|L|}_g\geq h^0(\omega_{\tilde{C}}\otimes\eta_{\tilde{C}})= \begin{cases}
        g-1$ if $\eta_{\tilde{C}}\neq\omega_{\tilde{C}}   \\
        g$ if $\eta_{\tilde{C}}=\omega_{\tilde{C}}
        \end{cases}$
\end{center}
The aforementioned result by Ciliberto, Dedieu, Galati and Knutsen states exactly that the latter is in fact an equality when $C$ is nodal. Since by the results due to Dedieu and Sernesi (see \cite{DS}) the equigeneric Severi varieties of genus $g$ curves in a K3 surfaces have dimension $g$, we actually have that equality also holds for nonnodal curves on $Y$ that split in the K3 cover. In other words, an irreducible component $V$ of an equigeneric Severi variety of curves of genus $g$ in an Enriques surface $Y$ such that the curves in $V$ split in the K3 cover has dimension $g$. Notice that we cannot say anything about the dimension of an irreducible component of an equigeneric Severi varieties of curves whose general member is doubly covered by an irreducible curve on the K3 cover. This discussion leads us to give the following definition.
\begin{definition}We shall call \textit{nonregular} a component of an equigeneric Severi variety of curves on an Enriques surface such that every curve splits in the K3 cover.\end{definition} 
The following lemma states that the fact that a nonregular curve on a very general Enriques surface $Y$ splits in the K3 cover in two linearly equivalent curves holds in general and not only for nodal curves.
 \begin{lemma}\label{tauvergen}
    Let $Y$ be a Picard very general Enriques surface and $X$ be is its universal cover. Moreover, let $\tau$ denote the Enriques quotient. Then, $\tau$ acts as the identity on $\ns(X)\cong U(2)\oplus E_8(-2)$.
\end{lemma}
\begin{proof}
    This is a very well-known fact. See, for example, \cite[Section 11]{DK3}.
\end{proof}
The previous lemma justifies the following definition
\begin{definition}
    We shall call \textit{special nonregular} a nonregular component of an equigeneric Severi variety of curves on an Enriques surface such that its members split in two nonlinearly equivalent curves in the K3 cover.
\end{definition}
%Part of the ideas we will use to show the existence of some special nonregular components of some equigeneric Severi varieties of curves are similar to the ones developed by Ciliberto and Dedieu in \cite{CiD} and presented by Dedieu in \cite[Section 5.2]{De}.

Starting from curves on $Y_m$ and $S$, we exploit Proposition \ref{partenzares} and Proposition \ref{partenzaenr} to show the existence of countable sets of special nonregular components of equigeneric Severi varieties of curves on $Y_m$.\par

Let $C$ be an irreducible effective divisor on $S$. As in section \ref{sec2}, we denote by by $C_{X,m}^{\boxplus k}$ the curve $g^{-1}(C)\boxplus R_m^{\boxplus k}$ and by $C_Y^\boxplus k$ its image $f(C_{X,m}^{\boxplus k})\subset Y_m$. Consider the complete linear system $|C_{X,m}^{\boxplus k}|$.

\begin{definition}
    We denote by \begin{center}$|C|_Y^{\boxplus k}\subset f_*(|C_{X,m}^{\boxplus k}|)$ \end{center}the family of the divisors on $Y_m$ which are images of curves on the linear system $|C_{X,m}^{\boxplus k}|$. \\
    We denote by $V_{Y,C,k}$ the subfamily of $|C|_Y^{\boxplus k}$ parametrizing curves which are images of irreducible curves on $|C_{X,m}^{\boxplus k}|$.
\end{definition}

We want to show that $V_{Y,C,k}$ is a special nonregular component of an equigeneric Severi varieties of curves for every irreducible effective divisor $C$ on $S$ different from an elliptic fiber.

\begin{lemma}\label{complete}
    Let $C$ be a smooth irreducible effective divisor on $S$. Then, the general curve of $|C_{X,m}^{\boxplus k}|$ is smooth and irreducible.
\end{lemma}

\begin{proof}
If $C$ is an elliptic fiber, then $|C_{X,m}^{\boxplus k}|$ coincides with the elliptic pencil on $X_m$ induced by the one of $S$ and therefore we have nothing to prove.
   If $C$ is different from an elliptic fiber, it splits in the K3 cover if and only if it intersects the branch locus $S_{t_0}+S_{t_{\infty}}$ with even order at every intersection point and this is not the case for the general curve of a complete linear system $|C|$. Hence, the general member of $|C_{X,m}^{\boxplus k}|$ is also irreducible. 
\end{proof}

\begin{theorem}\label{equifromres}
    Let $C$ be an irreducible effective $l$-section on $S$ of arithmetic genus $p$. Then, $V_{Y,C,k}$ is a special nonregular component of a Severi varieties of curves of genus $2g+l-1$.
\end{theorem}

\begin{proof}
By Proposition \ref{partenzares}, the preimage of a member $C_Y^{\boxplus k}\in V_{Y,C,k}$ consists of the two curves $C_{X,m}^{\boxplus k}$ and $C_{X,m}^{\boxplus (1-k)}$. In order to prove that $V_{Y,C,k}$ is nonregular, we need to show that $C_{X,m}^{\boxplus k}$ and $C_{X,m}^{\boxplus (1-k)}$ do not coincide. Suppose $C_Y^{\boxplus k}=C_{X,m}^{\boxplus (1-k)}$: for every $t$ and for every $j\in\{1,\dots,r\}$, we have $Q_{j,t}\boxplus R_m^{\boxplus k}\in C_{X,m}^{\boxplus (1-k)}$ and, by iterating, $Q_{j,t}\boxplus R_m^{\boxplus pk}\in C_{X,m}^{\boxplus (1-k)}$ for every $p\in\mathbb{Z}_+$. This is a contradiction because the number of intersection points between $C_{X,m}^{\boxplus (1-k)}$ and a fiber $X_t$ is finite and by Lemma \ref{inford} $R_m$ has infinite order in $\mw(X_m)$. \par Now, for a general fiber $X_t$ on $X_m$, we set \begin{center} ${C_X}_{|X_t}=a_1Q_{1,t}+\dots+a_rQ_{r,t}$\end{center} and \begin{center} ${C_X}_{|X_{-t}}=a_1Q_{1,-t}+\dots+a_rQ_{r,-t}$.\end{center} and we suppose $C_{X,m}^{\boxplus k}$ and $C_{X,m}^{\boxplus (1-k)}$ to be linearly equivalent. It follows that for the general $t\in\mathbb{P}^1$, the degree $a_1+\dots+a_r$ divisors  $a_1Q_{1,t}\boxplus R_m^{\boxplus k}+\dots+a_rQ_{r,t}\boxplus R_m^{\boxplus k}$ and $a_1Q_{1,t}\boxplus R_m^{\boxplus (1-k)}+\dots+a_rQ_{r,t}\boxplus R_m^{\boxplus (1-k)}$ are linearly equivalent in the elliptic curve $X_t$. In other words, we have the equality \begin{center}
    $Q_{1,t}^{\boxplus a_1}\boxplus R_{m_{t}}^{\boxplus k}\boxplus\dots\boxplus Q_{r,t}^{\boxplus a_r}\boxplus R_{m_{t}}^{\boxplus k}=Q_{1,t}^{\boxplus a_1}\boxplus R_{m_{t}}^{\boxplus (1-k)}\boxplus\dots\boxplus Q_{r,t}^{\boxplus a_r}\boxplus R_{m_{t}}^{\boxplus (1-k)}$\end{center} which is equivalent to
    \begin{center}
    $Q_{1,t}^{\boxplus a_1}\boxplus\dots\boxplus Q_{k,t}^{\boxplus a_k}\boxplus R_{m_{t}}^{\boxplus rk}=Q_{1,t}^{\boxplus a_1}\boxplus\dots\boxplus Q_{k,t}^{\boxplus a_k}\boxplus R_{m_{t}}^{\boxplus r(1-k)}$, \end{center} which implies that $R_m$ is a $r(2k-1)$-torsion section. Since by Lemma \ref{inford} $R_m$ has infinite order in $\mw(X_m)$, this is a contradiction.\par
Since $V_{Y,C,k}$ is nonregular, its dimension equals the geometric genus of any curve $C_Y^{\boxplus k}$ it parametrizes. The restriction $f_{|C_{X,m}^{\boxplus k}}$ is birational, so that the geometric genus of $C_Y^{\boxplus k}$ is equal to the arithmetic genus $p_a(C_{X,m}^{\boxplus k})$. Since $C$ is a $l$-section, it intersects the branch locus $S_{t_0}+S_{t_{\infty}}$ of $g:X_m\rightarrow S$ in $2l$ points. Hence, by the Riemann-Hurwitz formula, the arithmetic genus of $C_{X,m}^{\boxplus k}$ is \begin{center}
    $p_a(C_{X,m}^{\boxplus k})=2p+l-1$
\end{center}and this proves the assertion.

\end{proof}

We are able to produce special nonregular components also starting from smooth curves on $Y_m$.\par
Let $L$ be an irreducible effective divisor on $Y_m$. As in section \ref{sec2}, we denote by by $L_{X,m}^{\boxplus k}$ the curve $f^{-1}(L)\boxplus R_m^{\boxplus k}$ and by $L_Y^\boxplus k$ its image $f(L_{X,m}^{\boxplus k})\subset Y_m$. Consider the complete linear system $|L_{X,m}^{\boxplus k}|$.

\begin{definition}
    We denote by \begin{center}$|L|_Y^{\boxplus k}\subset f_*(|L_{X,m}^{\boxplus k}|)$ \end{center}the family of the divisors on $Y_m$ which are images of curves on the linear system $|L_{X,m}^{\boxplus k}|$. \\
    We denote by $V_{Y,L,k}$ the subfamily of $|L|_Y^{\boxplus k}$ parametrizing curves which are images of irreducible curves on $|L_{X,m}^{\boxplus k}|$.
\end{definition}
\begin{lemma}
    Let $L$ be a smooth irreducible effective divisor on $Y_m$. Then, the general curve of $|L_{X,m}^{\boxplus k}|$ is smooth and irreducible.
\end{lemma}
\begin{proof}
    It is well-known that the pullback of every genus 1 pencil on an Enriques surface via the Enriques quotient is a genus 1 pencil on its K3 cover, so that for $p_a(L)=1$ we have nothing to prove. If $p_a(L)\geq 2$, the preimage of $L$ is also smooth and irreducible: otherwise, the intersection between the two components of its preimage would be sent to nodes of $L$ and this is a contradiction.
\end{proof}

\begin{theorem}\label{equifromenr}
    Let $L$ be an effective divisor on $Y_m$ of arithmetic genus $g$. Then, $V_{Y,L,k}$ is a special nonregular component of a Severi varieties of curves of genus $2g-1$.
\end{theorem}
\begin{proof}
    The proof is essentially the same as the proof of Theorem \ref{equifromres}. The only difference is the geometric genus of the curves on $V_{Y,L,k}$. Since the quotient $f$ is base point free, by the Riemann-Hurwitz formula we have that \begin{center}
        $p_a(L_{X,m}^{\boxplus k})=2g-1$.
    \end{center} 
\end{proof}
We are able to prove Theorem \ref{teorenr}.
\begin{proof}[Proof of Theorem \ref{teorenr}]
    Notice that starting from a genus $g$ curve $L\subset Y_m$, the geometric genus of $L_Y^{\boxplus k}$ is $2g-1$, so that we have the assertion for odd numbers. For every even number $n=2z$, we need to find a $l$-section on $S$ of genus $p$, such that $2p+l-1=n$. But it is provided, for example, by the linear system \begin{center}
    $kL-(k-1)E_1$
\end{center}
on $S$, which is a $3k-(k-1)=2k+1$ section of genus $0$. Hence, we have $2p+l-1=0+2k+1-1=2k=n$.

\end{proof}

\subsection{\rm S\sc pecial superabundant components of logarithmic \rm S\sc everi varieties}
We introduce the so-called \textit{logarithmic Severi varieties}, parametrizing nodal curves with given tangency conditions to a fixed curve. The definition and the main results are given by Dedieu in \cite{De} and they are based on the works of Caporaso and Harris (see for example \cite{CH}). \par

Let us denote by $\underline{N}$ the set of all the sequences $\alpha=[\alpha_1,\alpha_2,\dots]$ of nonnegative integers with all but finitely many $\alpha_i$ non-zero. In practice we shall omit the infinitely many zeroes at the end. For a sequence $\alpha\in\underline{N}$, we let \begin{center}
    $|\alpha|=\alpha_1+\alpha_2+\dots$\\
    $\mathcal{I}\alpha=\alpha_1+2\alpha_2+\dots+n\alpha_n+\dots$
\end{center}

\begin{definition}\label{logsev}
Let $S$ be a smooth projective surface, $T\subset S$ a smooth, irreducible curve and $L$ be a line bundle or a divisor class on $S$ with arithmetic genus $p$. Let $\gamma$ be an integer satisfying $0\leq\gamma\leq p$, let $\alpha\in\underline{N}$ such that \begin{center}
    $\mathcal{I}\alpha=L\cdot T$.
\end{center} We denote by $V_{\gamma,\alpha}(S,T,L)$ the locus of curves in $L$ such that \begin{itemize}
    \item $C$ is irreducible of geometric genus $\gamma$ and algebraically equivalent to $L$,
    \item denoting by $\mu:\tilde{C}\rightarrow S$ the normalization of $C$ composed with the inclusion $C\subset S$, there exist $|\alpha|$ points $Q_{i,j}\in C$, $1\leq j\leq\alpha_i$ such that \begin{center}
        $\mu^*T=\sum\limits_{1\leq j\leq\alpha_i}iQ_{i,j}$.
    \end{center}
\end{itemize}
\end{definition}

\begin{theorem}[Dedieu]\label{dedieu}
    Let $V$ be an irreducible component of $V_{\gamma,\alpha}(S,T,L)$, $[C]$ a general member of $V$ and $\mu:\tilde{C}\rightarrow S$ its normalization as in the Definition \ref{logsev}. Let now $Q_{i,j}$, $1\leq j\leq\alpha_i$ points in $\tilde{C}$ such that \begin{center}$\mu^*T=\sum\limits_{1\leq j\leq\alpha_i}iQ_{i,j}$ \end{center} and set \begin{center}$D=\sum\limits_{1\leq j\leq\alpha_i} (i-1)Q_{i,j}$.\end{center}
    \begin{itemize}
        \item[(i)] If $-K_S\cdot C_i-$deg $\mu_*D_{|C_i|}\geq 1$ for every irreducible component $C_i$ of $C$, then \begin{center}
            $\dim(V)=-(K_S+T)\cdot L+\gamma-1+|\alpha|$
        \end{center}
        \item[(ii)] If $-K_S\cdot C_i-$deg $\mu_*D_{|C_i|}\geq 2$ for every irreducible component $C_i$ of $C$, then \begin{itemize}
            \item[(a)] the normalization map $\mu$ is an immersion, except possibly at the points $Q_{i,j}$;
            \item[(b)] the points $Q_{i,j}$ of $\tilde{C}$ are pairwise distinct.
        \end{itemize}

    \end{itemize}
\end{theorem}
According to \cite{CiD} and \cite{De}, we give the following definition.
\begin{definition}
    We call \textit{superabundant} an irreducible component $V$ of a logarithmic Severi variety of curves with $\dim(V)\geq -(K_S+T)\cdot L+\gamma+|\alpha|$
\end{definition}
In \cite[Section 5]{CiD} and \cite[Section 5]{De}, Ciliberto and Dedieu show the existence of superabundant logarithmic Severi varieties of curves on the plane $\mathbb{P}^2$. In their examples, they consider double covers of the plane and some totally tangent curves to the corresponding branch locus, which split in linearly equivalent curves in the double covering. These curves belong to the pushforward of the pullback under the double cover of linear systems of curves of a fixed degree on the plane: for this reason, the splitting curves are always $2$-divisible in $\pic(\mathbb{P}^2)$. Despite our examples arise in a similar context, the logarithmic Severi varieties we analyze behave differently with respect to the ones studied by Ciliberto and Dedieu. We study curves on a rational elliptic surface $S$ which splits in nonlinearly equivalent curves the K3 surface of base change type $X_m$ and they are not necessarily $2$-divisible in $\pic(S)$.
\begin{definition}
Let $R$ be a smooth surface, $B$ a curve on $R$ and let $T\rightarrow R$ be the double cover of $R$ branched at $B$.\\
    We call \textit{special} a superabundant component of a logarithmic Severi variety of curves on $R$ if it has dimension greater than expected and parametrizes curves which splits in two nonlinearly equivalent curves on $T$.
\end{definition}

The non base change very general K3 surfaces of base change type provide lots of examples of special components on rational elliptic surfaces. \par
Recall once more the main diagram
\begin{center}
                    \[\begin{tikzcd}[ampersand replacement=\&,row sep=large,column sep=huge]
                       \& X_m\arrow[dl,"g"] \arrow[dr,"f"]\\ S\&\&Y_m
                    \end{tikzcd}\]
\end{center} 
Let $C$ be a smooth irreducible $l$-section on $S$ of genus $p$. By Proposition \ref{partenzares}, the curves in the complete linear system $|C_X^{\boxplus k}|$ are identified with the curves on $|C_X^{\boxminus k}|$ by $g$. We call $V_{S,C,k}$ the family of the images on $S$ under $g$ of the irreducible curves inside $|C_X^{\boxplus k}|$. 
\begin{proposition}\label{logpartenzares}
    The family $V_{S,C,k}$ is a special superabundant component of a logarithmic Severi varieties of curves on $S$ of geometric genus $2p+l-1$.
\end{proposition}
\begin{proof}
The curves inside $V_{S,C,k}$ are $2l$-sections for $S$. Since they split on $X_m$, they are totally tangent to the branch locus $S_{t_0}\cup S_{t_{\infty}}$ in $l$ (possibly coincident) points for each fixed fiber. Their geometric genus is equal to the arithmetic genus of $C_X^{\boxplus k}$, which is $2p+l-1$ (see the proof of Proposition \ref{partenzares}). Equivalently, they belong to the logarithmic Severi variety $V:=V_{2p+l-1,[0,2l]}(S,S_{t_0}\cup S_{t_{\infty}},C_S^{\boxplus k})$. The expected dimension of $V$ is
\begin{center}
    $-(K_S+2F)\cdot C_S^{\boxplus k}+2p+l-1-1+2l=-2l+2p+l-2+2l=2p+l-2$,
\end{center}while the actual dimention of $V_{S,C,k}$ is \begin{center}
    $\dim(V_{S,C,k})=\dim(|C_X^{\boxplus k}|)=p_a(C_X^{\boxplus k})= 2p+l-1$.
\end{center}
\end{proof}
Let now $L\subset Y_m$ be a smooth curve of genus $g$. We call $V_{S,L,k}$ the family of the images on $Y_m$ under $f:X_m\rightarrow Y_m$ of the irreducible curves in the complete linear system $|L_X^{\boxplus k}|$
With the same arguments used for proving Proposition \ref{logpartenzares}, one proves the following proposition.
\begin{proposition}\label{logpartenzaenr}
    The family $V_{S,L,k}$ is a special superabundant component of a logarithmic Severi varieties of curves on $S$ of geometric genus $2g-1$.
\end{proposition}

\begin{remark}
    The existence of the special logarithmic Severi varieties $V_{S,C,k}$ and $V_{S,L,k}$ is due to the presence on $S$ of the rational bisection $B_{S,m}$ and then to the choice of the nongeneral pair ($S_{t_0},S_{t_\infty})$ of fixed fibers. In other words, the non-base change very generality of $X_m$ is crucial for the existence of special superabundant logarithmic Severi varieties of curves on $S$.
\end{remark}
The proof of Theorem \ref{teorres} is the same as the proof of Theorem \ref{teorenr} given at the end of the previous subsection.

  \end{document}